\documentclass{amsart}

\usepackage{graphics}

\theoremstyle{plain}
\newtheorem{theorem}{Theorem}[section]
\newtheorem{lem}[theorem]{Lemma}
\newtheorem{prop}[theorem]{Proposition}
\newtheorem{conjecture}[theorem]{Conjecture}

\theoremstyle{remark}
\newtheorem{question}[theorem]{Question}
\theoremstyle{definition}
\newtheorem{definition}[theorem]{Definition}
\newtheorem{example}[theorem]{Example}

\numberwithin{equation}{section}

\newcommand{\PMST}{\textnormal{Prob}_{\textnormal{MST}}}
\newcommand{\N}{\mathbb{N}}
\newcommand{\Z}{\mathbb{Z}}
\newcommand{\ra}{\rightarrow}
\newcommand{\R}{\mathbb{R}}
\newcommand{\E}{\mathbb{E}}
\newcommand{\G}{\mathcal{G}}
\newcommand{\V}{\mathcal{V}}
\newcommand{\mE}{\mathcal{E}}
\newcommand{\mA}{\mathcal{A}}
\newcommand{\mT}{\mathcal{T}}
\newcommand{\mF}{\mathcal{F}}
\newcommand{\mb}{\mathfrak{b}}
\newcommand{\tP}{\tilde{P}}
\newcommand{\Var}{\text{Var}}
\DeclareMathOperator{\cycle}{Cycle}
\DeclareMathOperator{\cut}{Cut}

\begin{document}

\title{On the minimum spanning tree distribution in grids}
\author {Kristopher Tapp}

\begin{abstract}
We study the minimum spanning tree distribution on the space of spanning trees of the $n$-by-$n$ grid for large $n$.  We establish bounds on the decay rates of the probability of the most and the least probable spanning trees as $n\ra\infty$, and we develop general tools for studying the decay rates of spanning tree families.
\end{abstract}

\maketitle

\section{Introduction}
Let $G=G(n)$ denote the $n$-by-$n$ square grid graph, which has $n^2$ vertices, and let $\mathcal{T}(G)$ denote the set of all spanning trees of $G$.  The MST (\emph{minimal spanning tree}) distribution on $\mathcal{T}(G)$ assigns to each $T\in\mathcal{T}(G)$ the probability, $\PMST(T)$, that $T$ will result from the algorithm that builds a spanning tree one edge at a time with each next edge uniformly randomly selected, along the way rejecting each addition that would create a cycle.  This is equivalent to applying Kruskal's algorithm to find a minimum spanning tree of $G$ with respect to uniformly identically distributed edge weights in $[0,1]$.

One purpose of this paper is to better understand the support of the MST distribution on $\mathcal{T}(G(n))$ asymptotically as $n\ra\infty$.  For this, we define a \emph{family of spanning trees}, $\mathcal{F}$, to mean a choice of a spanning tree $T_n\in\mathcal{T}(G(n))$ for each\footnote{All definitions and results will generalize effortlessly to a family that is only defined for an infinite subset of $\N$, like the fractal family in Figure~\ref{F:families}.} $n\in\N$.  Figure~\ref{F:families} shows three particular families that we will study.  We will first prove the following.

\begin{theorem}\label{T:bound2} For any family $\mathcal{F}=\{T_n\in\mathcal{T}(G(n))\mid n\in\N\}$,
$$ \PMST(T_n) \geq \left(\frac{1}{4}+o(1)\right)^{n^2} \,\text{ and }\,  \PMST(T_n) \leq \left(\frac{1}{2}+o(1)\right)^{n^2}.$$
\end{theorem}

We define the \emph{decay base}, $Q=Q(\mF)$, of a family $\mF$ as the value such that
$$ \PMST(T_n) =  \left(Q+o(1)\right)^{n^2}, \text{ or equivalently, } \lim\frac{\ln\left(\PMST(T_n)\right)}{n^2}=\ln Q,$$
provided such a value exists.  Even when such a value does not exist, we can define $Q^-=Q^-(\mF)$ and $Q^+=Q^+(\mF)$ as the infimum and supremum of such values, so Theorem~\ref{T:bound2} says that $Q^-(\mF)\geq 1/4$ and $Q^+(\mF)\leq 1/2$ for every family $\mF$.  In particular, $Q(\mF)\in\left[\frac 1 4,\frac 1 2\right]$ for any family $\mF$ for which $Q(\mF)$ exists.

To calibrate our understanding of Theorem~\ref{T:bound2}, recall from~\cite{Wu} that $|\mT(G(n))|$ grows asymptotically as $\mathfrak{b}^{n^2}$, where $\mathfrak{b}$ is related to Catalan's constant $C$ as follows:
$$\mb=\exp(4C/\pi)= 3.2099\cdots.$$
The average value of $\PMST(T)$ over all trees $T\in\mathcal{T}(G(n))$ equals $\frac{1}{|\mT(G(n))|}$.  If we consider a family $\mF$ is ``average'' in the sense that $\PMST(T_n)$ approximately equals $\frac{1}{|\mT(G(n))|}$ for each $n$, then the decay base of this ``average'' family equals $\frac{1}{\mathfrak{b}}$.

\begin{figure}[ht!]\centering
   \scalebox{.85}{\includegraphics{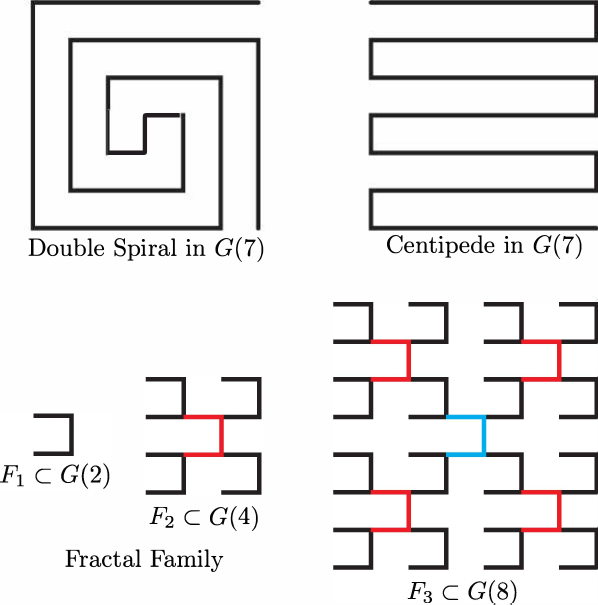}}
\caption{Three spanning tree families.  For the \emph{double spiral} and \emph{centipede} family, only the tree in $G(7)$ is shown, but the pattern extends in the obvious way.  For the \emph{fractal} family, introduced in~\cite{Alon}, each quadrant of $F_{k+1}$ is a copy of $F_{k}$, and these copies are connected in the middle by a copy of $F_1$.}\label{F:families}
\end{figure}

One purpose of this paper is to begin developing tools for computing or at least bounding the decay bases of particular families like those in Figure~\ref{F:families}.  Our main result along these lines is Theorem~\ref{T_pre}, which will establish that, for any family $\mF$ satisfying some natural hypotheses,
\begin{equation}\label{E:trailor}
Q^-(\mF) \geq\frac{1}{e\cdot\overline{f}},
\end{equation}
where $\overline{f}=exp\left(\int_0^1\ln f(x)\,\text{dx}\right)$ is the geometric mean of the power series $f$ that in Definition~\ref{D:pre_powerseries} we will naturally associate to $\mF$.  More precisely, the $k^\text{th}$ coefficient of $f$ reports, among edges of $G(n)$ not in $T_n$ for large $n$, the limiting proportion that make a cycle of length $k+1$ when added to $T_n$.

This paper is organized as follows.  In Section 3,  we introduce a construction that associates a bipartite graph to a spanning tree.  Among other things, this construction helps partially account for an empirical observation from the mathematics of redistricting, namely that spanning trees drawn from the MST distribution tend to result in more compact district maps.

In Section 4, we discuss the Lyons-Peres formula for the MST probability of a spanning tree, which turns out to depend only on the associated bipartite graph.  Because of this, many of the propositions in this paper could be framed as results about bipartite graphs.  We then prove Theorem~\ref{T:bound2} in Section 5.

In Section 6, we state Theorem~\ref{T_pre}, which was informally summarized in Equation~\ref{E:trailor}.  The proof of this theorem is involved and is spread through Sections 7-9.  In Section 10, we study the example families of Figure~\ref{F:families}.  Finally, in Section 11 we provide some questions for further study and a conjecture about the sharpness of Equation~\ref{E:trailor}.  More specifically, the only reason that Equation~\ref{E:trailor} is a lower bound and not an equality is a single application of Jensen's inequality.  We conjecture an \emph{exact} value for the decay base $Q(\mF)$, and sketch a possible proof of this conjecture that relies on controlling the Jensen gap.

\section*{Acknowledgments} The author is pleased to thank Moon Duchin for helpful suggestions, Misha Lavrov for assistance via Stack Exchange, and the referees for valuable comments.
\section{Related work} 
There is a limited body of work relating the MST and UST (uniform spanning tree) distributions.  The authors of~\cite{complete} show that the expected diameter of a MST
of the complete graph on $n$ vertices is on the order of $n^{1/3}$, whereas the expected diameter of a UST is $n^{1/2}$.  Chapter 10 of~\cite{Russ_book} studies the MST distribution in a setting generalized to infinite graphs.  The authors of~\cite{Russ} derived an explicit formula for $\PMST(T)$.  This formula is the basis of most of the results of our paper.  Very recently, the authors of~\cite{Models} derived a dual version of this formula.  We will include both the formula and its dual version in Section 4.  The authors of~\cite{Models} also provide new tools for studying the MST distribution and generalizations in which the edge weights are not i.i.d but rather are drawn from separate distributions. 

Similar to~\cite{Models}, one of the motivations of our paper is the application to redistricting. The starting point of modern redistricting algorithms is a tiling of a state into precincts.  The tiling's dual graph $\mathcal{G}$ has a vertex for each precinct and an edge between each pair of vertices that correspond to adjacent precincts.  A key step in modern sampling algorithms is to draw a random spanning tree on a certain subgraph of $\mathcal{G}$.  The methods in~\cite{SMC} and~\cite{MergeSplit} use Wilson's algorithm to draw a UST, whereas the method in~\cite{ReCom} by default uses Kruskal's algorithm to draw a MST.  It was empirically demonstrated in~\cite{Rev_ReCom} that MST-based methods result in plans with more compact districts (more specifically, with fewer cut edges).  This empirical data has not been fully explained, and is one motivation for this paper.

There is a fairly large body of literature on a question that is perhaps more than tangentially related to ours.  The seminal paper~\cite{Alon} defined the \emph{average stretch} of a spanning tree $T$ of a graph $G=(V,E)$ as 
\begin{equation}\label{E:avg_stretch}
\text{avg-stretch}(T) = \text{avg}\{\text{dist}_T(e)\mid e\in E\},
\end{equation}
where $\text{dist}_T(e)$ denotes the path-length distance in $T$ between the two endpoints $e$ (which equals $1$ if $e\in T$).  This paper and the follow-up literature focused on minimizing this measurement, with special attention to the case of grid graphs~\cite{KLRW},\cite{Koh},\cite{Lieb}.  In fact, the fractal family from Figure~\ref{F:families} was introduced in~\cite{Alon} as an example of a family with low average stretch. 

The relevance is that $\PMST$ is negatively correlated with the average stretch of $T$; that is, a tree with lower average stretch tends to have a higher MST probability.  Although we are not aware of any rigorous bounds along these lines, Section 4 will mention an analytic reason to expect this correlation, and Figure~\ref{F:scatter_stretch_mst} will provide empirical evidence.  Furthermore, the two optimization problems have the same solutions in the complete graph $K_n$.  It was recently proven in \cite{Models} that a star has the highest MST probability in $K_n$ while a path has the lowest.  On the other hand, it is straightforward to show that a star minimizes avg-stretch while a path maximizes it.  

For grid graphs, these problems seem harder.  Identifying the spanning tree of $G(n)$ that minimizes avg-stretch is an open problem for $n>10$, and even less is known about the MST-maximizer.  Asymptotic bounds as $n\rightarrow\infty$ may be more obtainable than exact optima for particular $n$,  which motivates the work in this paper to obtain bounds on the decay bases of families of spanning trees.
\section{The Bipartite Graph}
In this section, we introduce a simple construction that naturally associates a bipartite graph to a spanning tree.

Let $G=(V,E)$ be a connected simple undirected graph (not necessarily a grid).  Let $T$ be a spanning tree of $G$ regarded as a subset of $E$.  We'll call the edges of $T$ \emph{branches} and the edges of $T^c = E-T$ \emph{chords}. For each chord $e\in T^c$, let $\cycle(e)$ denote the unique cycle in the graph obtained by adding $e$ to $T$.  For each branch $e\in T$, let $\cut(e)$ denote the cut set between the two components of $T-\{e\}$; that is, $\cut(e)$ includes $e$ and each chord that's incident to both components.

We define a bipartite graph $\G=(\V_1,\V_2,\mE)$ associated to the spanning tree $T$ as exemplified in Figure~\ref{F:bipartite}.  It contains a set $\V_1$ of nodes corresponding to the branches of $T$, and a set $\V_2$ of nodes corresponding to the chords, with adjacency determined by whether the branch belongs to the cycle associated to the chord.  Denote $M=|\V_1|=|T|$ and $N=|\V_2|=|T^c|$.  Notice that a branch $e\in\V_1$ is adjacent to all of the chords in $\cut(e)$, while a chord $e\in\V_2$ is adjacent to all of the branches in $\cycle(e)$.  To avoid confusion, we'll use the terms \emph{vertices} and \emph{edges} when referring to $G$, and the terms \emph{nodes} and \emph{links} when referring to $\G$.

\begin{figure}[ht!]\centering
   \scalebox{.9}{\includegraphics{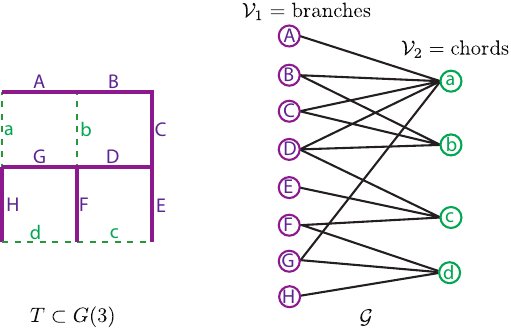}}
\caption{The bipartite graph $\G$ associated to a spanning tree $T\subset G(3)$}\label{F:bipartite}
   \end{figure}

For $i\in\{1,2\}$, let $d_i$ denote the average degree in $\G$ of the nodes in $\V_i$.  Notice that average stretch is a linear function of $d_2$ because Equation~\ref{E:avg_stretch} can be re-written as $\text{avg-stretch}(T) = \frac{M+N\cdot d_2}{M+N}$; thus, optimizing average stretch (among spanning trees of a fixed graph) is the same as optimizing $d_2$.

On the other hand, $d_1+1$ equals the expected number of cut edges between the two components obtained by removing a random branch from $T$.  This paper focuses on the case where $G=G(n)$ for large $n$, in which case\footnote{More precisely, any spanning tree $T$ of $G(n)$ has $M=n^2-1$ branches and $N=n^2-2n+1$ chords.}
$M\approx N$ so $d_1\approx d_2$.  In this setting, a tree with small average stretch (small $d_2$) will on average form a bipartition with fewer cut edges when a random one of its branches is removed (because $d_1$ is small).

In the literature on the mathematics of redistricting, it has been observed empirically that a spanning tree drawn from the MST distribution has exactly this advantage compared to a spanning tree drawn uniformly -- it tends to form a more compact bipartition when one of its branches is removed; see~\cite{revRecom}.  This observation is partially accounted for by the above comments together with the previously mentioned observation that a spanning tree with higher MST probability tends to have a lower average stretch.

\section{An Explicit Formula for $\PMST(T)$}
In this section, we state and re-interpret the explicit formula for MST probability due to Lyons and Peres~\cite{Russ}, together with a dual version of the formula from~\cite{Models}. We will re-frame both formulas in a way that makes it clear that $\PMST(T)$ depends only on the bipartite graph $\G$ associated to $T$.  In fact, that is the framing under which the duality of the two formulas is best understood. 

Exactly as in the previous section, let $G$ be a connected graph (not necessarily a grid), let $T$ be a spanning tree of $G$, let $\mathcal{G}=(\V_1,\V_2,\mathcal{E})$ be the associated bipartite graph, and define $M = |\V_1|=|T|$ and $N=|\V_2|=|T^c|$.  Let $S(T)$ denote the set of permutations (orders) of $T$.

\begin{definition}\label{D:PassingTimes} For $\alpha=(e_1,...,e_M)\in S(T)$ and $1\leq i\leq M$, define
$$P_i(\alpha) = |\V_1^i|+|\V_2^i| = i + |\V_2^i|,$$
where $\V_1^i = \{e_1,...,e_i\}\subset\V_1$, and $\V_2^i\subset\V_2$ denotes the set of nodes in $\V_2$ that are connected by an arc in $\G$ to at least one node in $\V_1^i$.
\end{definition}
This definition of $P_i(\alpha)$ can be illuminated via the following metaphor.  The nodes of $\V_2$ are unlit candles, while the nodes of $\V_1$ are people who wish to light the candles using a single lighter that they must share.  The people take turns in the order $\alpha$.  On his/her turn, a person lights each not-yet-lit candle to which he/she is connected by a link in $\G$.  Lighting one candle takes one unit of time, and passing the lighter to the next person also takes one unit of time.  In this story, $P_i(\alpha)$ equals the time at which the $i^\text{th}$ pass of the lighter occurs.  We will therefore refer to $(P_1(\alpha),...,P_M(\alpha))$ as the sequence of \emph{passing times} associated with $\alpha$.

\begin{prop}[Prop 5.1 of \cite{Russ}; see also Theorem 3.5 of~\cite{Models}]\label{P:russ}
$$\PMST(T) = M!\cdot \E \left(\frac{1}{\prod_{i=1}^M P_i(\alpha)}\right) = \sum_{\alpha\in S(T)} \frac{1}{\prod_{i=1}^M P_i(\alpha)}, $$
where the $\E $ denotes the expectation with respect to the uniform distribution on $S(T)$.
\end{prop}
\begin{proof}[Sketch of Proof]
The proof in~\cite{Russ} (a more complete version of which is provided in~\cite{Models}) involves noticing that for a fixed permutation $\alpha=(e_1,e_2,...,e_M)$,
$$\frac{1}{\prod_{i=1}^M P_i(\alpha)} = \frac{1}{P_M(\alpha)}\cdot\frac{1}{P_{M-1}(\alpha)}\cdots\frac{1}{P_1(\alpha)}$$
equals the probability that the algorithm described in the first paragraph of this paper will yield $T$ in the \emph{reverse} order $(e_M,...,e_1)$.  This is because, if $\{e_{i+1},...,e_M\}$ are the edges added to the tree so far, then $P_i(\alpha)=|\V_1^i|+|\V_2^i|$ is the number of possibilities for the next edge that wouldn't be rejected by the algorithm.  This is because $|\V_1^i|$ counts the not-yet-used branches, while $|\V_2^i|$ counts the chords that wouldn't be rejected.
\end{proof}
Imagine running Kruskal's algorithm and recording the number of available choices at each step (edges that wouldn't be rejected).  This is a decreasing list of numbers that starts with $N+M$.  The reversal of this list is the sequence of passing times.

\begin{example}
Let $T\subset G(3)$ be the spanning tree in Figure~\ref{F:bipartite} and let $\alpha\in S(T)$ be the alphabetical order as labeled; that is $\alpha=(A,B,...,H)$.  Here $M=8$ and $N=4$.  The passing times are:
$$\{P_i(\alpha)\} = \{2,4,5,7,8,10,11,12\}.$$
This can be derived using only the image of $\G$ in Figure~\ref{F:bipartite}, and it can be understood via the candle metaphor by indicating what happens at each of the $M+N=12$ time steps as follows.
\begin{center}
\begin{tabular}{ c c c c c c }
 1. Aa, & \textbf{2}. A$\ra$B, & 3. Bb,  & \textbf{4}. B$\ra$C, & \textbf{5}. C$\ra$D, & 6. Dc \\
 \textbf{7}. D$\ra$E, & \textbf{8}. E$\ra$F  ,& 9. Fd ,& \textbf{10}. F$\ra$G ,& \textbf{11}. G$\ra$H ,& \textbf{12}. H$\ra$
\end{tabular}
\end{center}
The coding here shows what happens at that time step: at time $1$ person $A$ lights candle $a$, then at time 2 person $A$ passes the lighter to person $B$, and so on.  Notice that $P_{8}=12$ is the last passing time, so we need to amend the candle metaphor by adding that the final person performs one final pass to put the lighter down.
\end{example}

The dual version of Proposition~\ref{P:russ} from~\cite{Models} says the following.
\begin{prop}[Theorem 3.4 of~\cite{Models}]\label{P:dual} Proposition~\ref{P:russ} remains true if the roles of $\V_1$ and $\V_2$ are reversed in Definition~\ref{D:PassingTimes}.  
\end{prop}
In the reversed version, $\V_2$ represents the people and $\V_1$ represents the candles.  More precisely, for $\alpha = (e_1,...,e_N)\in S(T^c)$ and $1\leq i\leq N$, the reversed version would define  
$P_i(\alpha)=|\V_2^i|+|\V_1^i| = i+|\V_1^i|$,
where $\V_2=(e_1,...,e_i)$, and $\V_1^i$ denotes the set of nodes in $\V_1$ that are connected by an arc in $\mathcal{G}$ to at least one node in $\V_2$.  With these definitions, it is true that 
$\PMST(T) = \sum_{\alpha\in S(T^c)} \frac{1}{\prod_{i=1}^N P_i(\alpha)}.$  

We will not use the dual version, so in this paper, the passing times will always be defined as in Definition~\ref{D:PassingTimes} (not the dual version defined in the previous paragraph).  However, it is worth mentioning that most of the results of this paper would remain valid with the roles of $\V_1$ and $\V_2$ reversed.

Figure~\ref{F:scatter_stretch_mst} suggests that $\PMST(T)$ is negatively correlated with the average stretch of $T$.  Proposition~\ref{P:russ} makes this quite believable; if the average degree in $\V_1$ is high, then there are more arcs in $\G$, which would be expected to lead to higher passing times (the candles are lit earlier) with respect to a random ordering of $\V_1$.

\begin{figure}[ht!]\centering
   \scalebox{.5}{\includegraphics{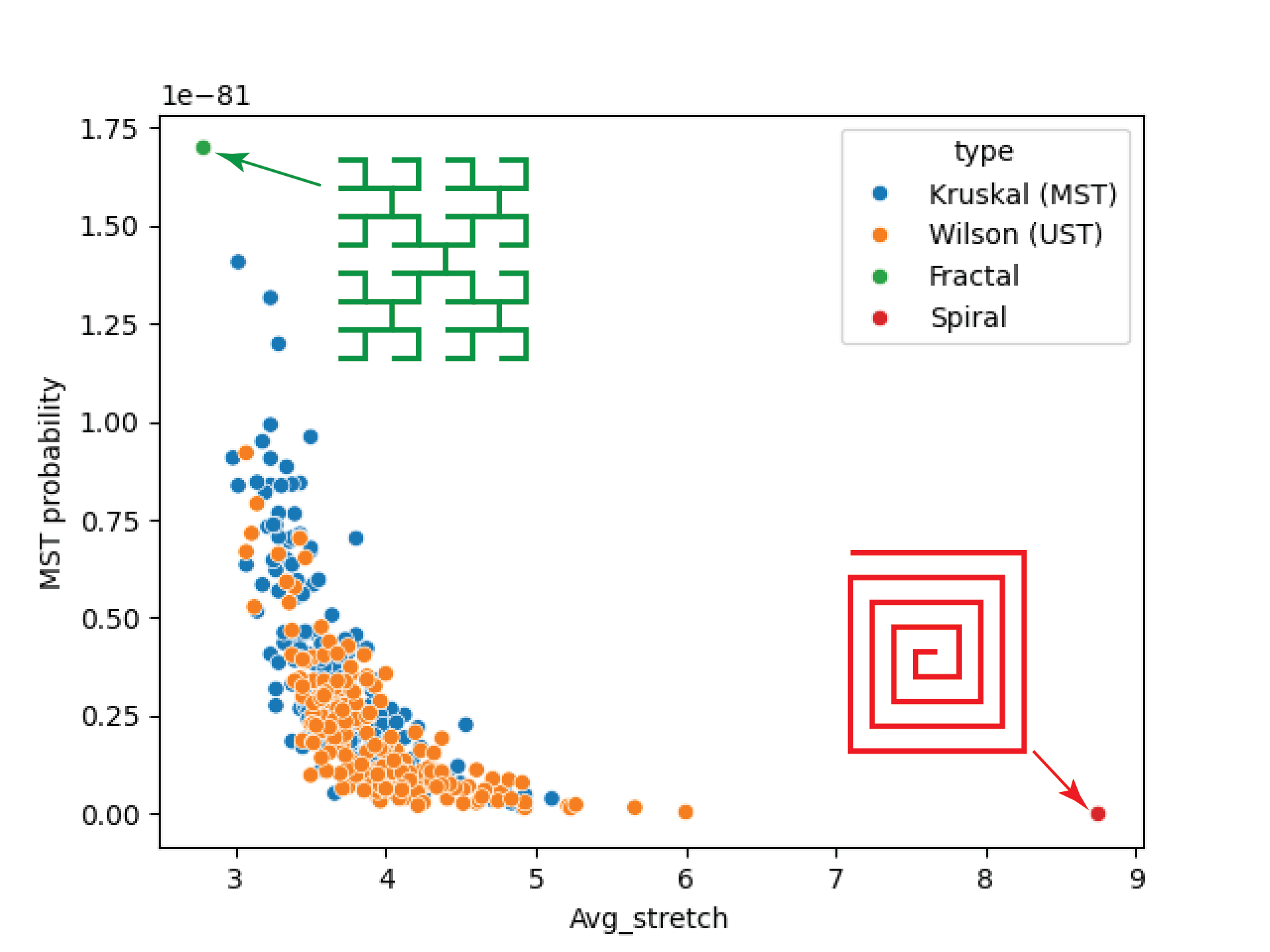}}
\caption{$\PMST$ versus avg-stretch for 2 special trees in $G(8)$ and 200 random trees (100 each from Kruskal's and Wilson's algorithm).  Here $\PMST$ is approximated via Proposition~\ref{P:russ} using a random sample of 10000 choices of $\alpha\in S(T)$.}\label{F:scatter_stretch_mst}
   \end{figure}
\section{Proof of Theorem~\ref{T:bound2}}
This section is devoted to the proof of Theorem~\ref{T:bound2}.  We begin by introducing notation for asymptotic growth.  If $f,g:\mathbb{N}\ra\R^+$ and $b,p\in\R^+$, then we define:
\begin{enumerate}
    \item $f(n)\approx g(n)$ means $\frac{f(n)}{g(n)}\rightarrow 1$.
    \item $f(n)<\approx g(n)$ means $\lim \frac{f(n)}{g(n)}\leq 1$.
    \item $f(n)\sim b^{n^p}$ means $f(n)=\left(b+o(1)\right)^{n^p}$, or equivalently,
    $\frac{\ln(f(n))}{n^p}\rightarrow \ln(b)$.
    \item $f(n)<\sim b^{n^p}$ means $f(n)\leq\left(b+o(1)\right)^{n^p}$, or equivalently,
    $$\limsup \frac{\ln(f(n))}{n^p}\leq\ln(b).$$   
    \item $f(n)>\sim b^{n^p}$ means $f(n)\geq\left(b+o(1)\right)^{n^p}$, or equivalently,
    $$\liminf \frac{\ln(f(n))}{n^p}\geq\ln(b).$$    
\end{enumerate}

\begin{proof}[Proof of Theorem~\ref{T:bound2}]
Any spanning tree $T$ of $G(n)$ has $M=n^2-1$ branches and $N=n^2-2n+1$ chords. In particular, $N<M<n^2$.

Proposition~\ref{P:russ} says that $\PMST(T) = M!\cdot \E \left(\frac{1}{\prod P_i(\alpha)}\right)$, so it will suffice to find lower and upper bounds for $\frac{M!}{\prod P_i(\alpha)}$ over all $T\in\mathcal{T}(G(n))$ and all $\alpha\in S(T)$.  For convenience, we'll instead work with its reciprocal $\frac{\prod P_i(\alpha)}{M!}$.

For any $T\in\mathcal{T}(G(n))$ and any $\alpha\in S(T)$, the passing times $\{P_i(\alpha)\}$ form a sequence of $M$ strictly increasing integers between $1$ and $M+N$.  We know that $P_i(\alpha)\leq N+i$ because the upper extreme possibility for the sequence is $\{N+1, N+2, ..., N+M\}$.  So Stirling's approximation gives:
\begin{align}\label{E:upperbound}
\frac{\prod P_i(\alpha)}{M!} & \leq \frac{(M+N)!}{N!M!} = {M+N \choose M} \leq { 2M \choose M} \approx \frac{4^M}{\sqrt{\pi M}}< 4^M <4^{n^2},
\end{align}
from which it follows that $\PMST(T)>\sim \left(\frac{1}{4}\right)^{n^2}$ as desired.

Next, to prove the other bound, we wish to establish that $2i$ is an approximate lower bound for $P_i(\alpha)$.  For this, consider the forest (acyclic graph) $F_i$ that has the same vertex set as $G(n)$ and has the edge set $\{e_{i+1},e_{i+2},...,e_{M}\}$.  This forest has $i+1$ components (some of which might be individual vertices).  A ``cut edge'' of $F_i$ means an edge of $G(n)$ that connects two different components of $F_i$.  Observe that $P_i(\alpha)$ equals the number of cut edges of $F_i$.  So we must prove that a forest with $i+1$ components has at least approximately $2i$ cut edges.  For this, notice that $J=4n$ equals the number of ``exterior connectors,'' which means edges that connect vertices of $G(n)$ to vertices of its complement in the infinite lattice $\Z^2$.  Every component of $F_i$ is incident to at least $4$ edges that are either cut edges or exterior connectors, namely at its top, bottom, left, and right-most vertices, so $4(i+1)\leq 2P_i(\alpha) + J$, or $P_i(\alpha)\geq 2(i+1)-\frac 12 J$.

Since $\frac J 2=2n$, we learn that $P_i(\alpha)\geq 2(i+1)-2n$ for all $i\in\{1,2,...,M\}.$
This gives no information about the indices $i\in\{1,2,...,n-1\}$, but applying the result to the remaining indices yields
$$\prod P_i(\alpha)\geq \prod_{i\geq n} P_i(\alpha) \geq 2\cdot 4\cdot 6\cdots (2M-2n).$$

We therefore have the following.
$$\frac{\prod P_i(\alpha)}{M!}\geq 2^{\left(M-n\right)} \cdot\frac{\left(M-n\right)!}{M!}
\geq \frac{2^{\left(M-n\right)}}{M^n}.$$
It follows that
$$\PMST(T) = \E\left(\frac{M!}{\prod P_i(\alpha)} \right)
\leq \frac{M^n}{2^{(M-n)}} 
= \frac{\left(n^2-1\right)^n}{2^{(n^2-1-n)}}
\sim \frac{1}{2^{n^2}},$$
so $\PMST(T)<\sim \left(\frac{1}{2}\right)^{n^2} $.
\end{proof}

\section{Statement of main result}
The next three sections will develop machinery that, among other things, will yield a lower bound on the decay base of certain types of families of spanning trees.  In this section, we preview these upcoming sections by stating and roughly explaining this bound, which we already mentioned informally in Equation~\ref{E:trailor}.  We begin by defining the additional hypotheses that will be needed for our spanning tree families.

\begin{definition}\label{D:pre_bounded} The family $\mathcal{F}=\{T_n\subset G(n)\mid n\in\N\}$ is \emph{bounded} if the average degree (among vertices in the bipartite graph associated to $T_n$) divided by $n^2$ goes to zero as $n\ra\infty$.
\end{definition}

For example, the fractal family satisfies a much stronger hypothesis than bounded; it was proven in~\cite{Alon} that the average degree grows as $C\cdot\ln(n)$ for this family.  The centipede family is also bounded, but the double spiral family is not.

\begin{definition}\label{D:pre_degreemass} The \emph{degree-mass} function of a spanning tree $T$ is the function $p:\N\ra[0,1]$ defined so that for all $d\geq 1$, $p(d)$ equals the portion of the nodes in $\V_2$ that have degree $d$ (in the bipartite graph associated with $T$).
\end{definition}

For a grid graph, the support of the degree-mass function $p$ includes only odd numbers that are $\geq 3$.  This is because the length of every cycle of $G(n)$ is an even number that is $\geq 4$.

\begin{definition}\label{D:pre_powerseries} Let $\mF = \{T_n\subset G(n)\mid n\in\N\}$ be a family of spanning trees.  Let $p_n$ denote the degree-mass function for $T_n$.  We call $\mF$ \emph{convergent} if for each $d$, $p_\infty(d)=\lim_{n\ra\infty}p_n(d)$ exists.  In this case, the \emph{power series associated to} $\mF$ is the function $f:[0,1]\ra\R$ defined as follows.
$$ f(x) = 1+x-\sum_{d\geq 3} p_\infty(d)\cdot (1-x)^d.$$
\end{definition}

The centipede, double spiral and fractal families are all convergent. For each of these three families, we'll compute (at least a truncation of) the associated power series in Section 10.  An example of a non-convergent family can be obtained by alternating between members of two different families, like say the centipede and fractal families.

We will require one more hypothesis, which is technical but is also weak.
\begin{definition}\label{D:neighborcon} Let $\mF = \{T_n\subset G(n)\mid n\in\N\}$ be a family of spanning trees.  For $n\in\N$, denote $\Delta_n=\V_2\times\V_2$ (in the bipartite graph associated to $T_n$), and let  $\tilde\Delta_n$ be the set of pairs $(v,w)\in\Delta_n$ that have at least one common neighbor in $\V_1$.  The family $\mF$ is called \emph{neighbor-independent} if
$\lim_{n\ra\infty}\frac{|\tilde\Delta_n|}{|\Delta_n|}=0$.
\end{definition}

The definition says that most pairs of nodes of $\V_2$ do not share a common neighbor.  Referring back to $T_n\subset G(n)$ for large $n$, this means that most pairs of chords complete non-overlapping cycles in $T_n$.  We will show in Section 10 that the fractal and centipede families are both neighbor-independent.

One purpose of the next three sections is to prove the following.

\begin{theorem}\label{T_pre}
If $\mF=\{T_n\subset G(n)\mid n\in\N\}$ is a bounded, convergent, neighbor-independent family, and $f$ is its associated power series, then its decay base is bounded as
$$Q^-(\mF) \geq\frac{1}{e\cdot\overline{f}},$$
where $\overline{f}=exp\left(\int_0^1\ln f(x)\,\text{dx}\right)$ is the geometric mean of $f$.
\end{theorem}

The basic idea will be to show that $f$ is the limit of the passing time sequences.  More precisely, for each $n$ one could uniformly choose $\alpha\in S(T_n)$ and then plot the sequence $\left\{\left(\frac i M,\frac{P_i(\alpha)}{M}\right)\right\}$, where $M=n^2-1=|T_n|$.  The reason for dividing by $M$ is to scale the plot from $[1,M]\times[1,2M]$ down to $[0,1]\times[0,2]$ so that these plots can be compared across different values of $n$.  In fact, as $n\ra\infty$, we'll show essentially that these plots look more and more like the graph of $f$ with higher and higher probability.  This intuition makes Theorem~\ref{T_pre} plausible, since Proposition~\ref{P:russ} can be used to relate $Q(\mF)$ to the geometric mean of the passing times.

\section{Approximate passing times}
In this section, we let $T\subset\mathcal{T}(G(n))$ and we derive an exact and an approximate formula for the expected value of an individual passing time.  Most of the results will also hold for not-necessarily-grid graphs.  Consider all running notation from Sections 3 and 4 to be defined with respect to $T$.  In particular, $p$ denotes the degree mass function of $T$.

\begin{prop}\label{P:Pavg} For each $i\in\{1,2,...,M\}$ we have:
\begin{align*}
\E (P_i(\alpha)) & = i + N\left(1-\sum_{d\geq 3} p(d)\cdot\frac{{M-d \choose i}}{{M \choose i}}  \right),
\end{align*}
where $\E$ denotes the expectation with respect to a uniform choice of $\alpha\in S(T)$.  Here the convention is that ${M-d \choose i}=0$ when $i>M-d$.
\end{prop}

\begin{proof}
Fix $i\in\{1,2,...,M\}$.  For each node $v\in\V_2$, let $d(v)$ denote the degree of $v$, and let $\phi(v)$ denote the probability (with respect to a uniform choice of $\alpha$) that $v\in\V_2^i$.  Note that $\phi(v)$ equals the probability that a random selection without replacement of $i$ of the nodes of $\V_1$ includes at least one of the $d(v)$ nodes which are adjacent to $v$.  Thus,
\begin{align*}
\E (P_i(\alpha))
   & = \E \left( i+|\V_2^i|\right) \\
   &  = i + \sum_{v\in \V_2} \phi(v)
     = i + \sum_{v\in \V_2} \left( 1-\frac{{M-d(v) \choose i}}{{M \choose i}}  \right)\\
   & = i + \sum_{d\geq 3} p(d)\cdot N\cdot \left( 1-\frac{{M-d \choose i}}{{M \choose i}}  \right),
\end{align*}
which simplifies to the claimed result.
\end{proof}

The formula in Proposition~\ref{P:Pavg} is difficult to manage, so we next seek to approximate this formula, at least for \emph{bounded} families.  For this, we will define a sequence $\{\tP_1,\tP_2,...,\tP_M\}$, called the \emph{approximate passing times}, such that $\tP_i$ approximates the formula for $\E(P_i(\alpha))$ in Proposition~\ref{P:Pavg}.

\begin{definition}\label{D:apt} For each $i\in\{1,2,...,M\}$, the $i^\text{th}$ approximate passing time is
$$\tP_i = i + N\left( 1 - \sum_{d\geq 3}p(d)\left(\frac{M-i}{M}\right)^d\right).$$
\end{definition}

For the next proposition, we will scale the actual and approximate passing times so they can be meaningfully compared across a family of spanning trees.  For every $i$ and every $\alpha$,  $P_i(\alpha)\leq M+N<2M$, so the \emph{scaled passing time} $P_i(\alpha)/M$ lies in $(0,2)$.  Similarly, the \emph{scaled approximate passing time} $\tP_i/M$ lies in $(0,2)$.

\begin{prop}\label{P:Approxpasstimes} For each $i\in\{1,2,...,M\}$, $$\E(P_i(\alpha))\geq\tP_i.$$  Moreover,
if $\mF=\{T_n\subset G(n)\mid n\in\N\}$ is a bounded family of spanning trees, then for every $x\in(0,1)$,
$$
\lim_{n\ra\infty}\E \left(\frac{P_{\lfloor x M_n\rfloor}(\alpha)}{M_n}\right)
 =\lim_{n\ra\infty}\left(\frac{\tP_{\lfloor x M_n\rfloor}}{M_n}\right),
$$
provided the limit on the right exists, where $M_n = n^2-1 = |T_n|$.
\end{prop}

Here $\lfloor\rfloor$ denotes the floor function.  If say $x=.15$, then the proposition considers, for each spanning tree in the family, the actual and approximate scaled passing time that is about $15\%$ through the sequence.

\begin{proof}
Let $x\in(0,1)$.  Consider a fixed large value of $n$, and interpret all running notation with respect to the spanning tree $T=T_n$; in particular, $M_n$ will be written simply as $M$, and $p$ will denote the degree mass function of $T$.  Set $i={\lfloor x M\rfloor}$.

Comparing the formulas from Proposition~\ref{P:Pavg} and Definition~\ref{D:apt}, we see that it will suffice to prove that the following ``error'' becomes arbitrarily small as $n\ra\infty$:
\begin{equation}\label{E:error}
E = \sum_{d\geq 3} p(d)\cdot \left|\left(\frac{M-i}{M}\right)^d
                           - \frac{{M-d \choose i}}{{M \choose i}} \right|.
\end{equation}

Let $\epsilon>0$.  Let $\mu$ denote the average degree in $\V_2$.  Markov's inequality gives that
$\sum_{d\geq \frac{2\mu}{\epsilon}}p(d)\leq\frac{\epsilon}{2}$.  Thus,
$$E \leq\frac{\epsilon}{2}+ \sum_{d< \frac{2\mu}{\epsilon}} p(d)\cdot \left|\left(\frac{M-i}{M}\right)^d
                           - \frac{{M-d \choose i}}{{M \choose i}} \right|. $$

Since $\mF$ is bounded, we can assume that $n$ is chosen so that $\mu$ is arbitrarily small relative to $\min\{i,M-i\}$.  In fact, $n$ can be chosen so that each value of $d$ for which $d\leq \frac{2}{\epsilon}\cdot\mu$ is arbitrarily small relative to $\min\{i,M-i\}$.

But for values of $d$ that are small relative to $\min\{i,M-i\}$, there is small error to the natural approximation of the probability, if an urn contains $M$ balls of which $d$ are red, that a selection of $i$ of them includes no red balls:
\begin{align*}
\frac{{M-d \choose i}}{{M \choose i}}
   = \left(\frac{M-i}{M}\right)\cdot\left(\frac{M-i-1}{M-1}\right)
      \cdots\left(\frac{M-i-d}{M-d}\right)
  \approx \left(\frac{M-i}{M}\right)^d.
\end{align*}
The previous equation would also remain true if the final ``$\approx$'' were replaced by ``$\leq$'', which justifies the assertion that $\E(P_i(\alpha))\geq\tP_i.$
\end{proof}

\section{The power series of a family of spanning trees}

Proposition~\ref{P:Approxpasstimes} allows us to approximate the expected values of individual passing times.  The proposition ends with the caveat ``provided the limit on the right exists.''  In this section, we show that for a \emph{convergent} family, this limit not only exists but is given by a simple formula, and we can also establish a bound on the \emph{variance} of an individual passing time, denoted $\Var(P_i(\alpha))$.

If $\mF$ is a convergent family and $f$ is its associated power series (as in Definition~\ref{D:pre_powerseries}), it is straightforward to show that $\sum p_\infty(d)\leq 1$, but we'll later see examples where this sum is strictly less than $1$, so $p_\infty$ is not necessarily a valid probability function.  Notice that $f'(x)\geq 1$ for all $x\in[0,1]$.  Moreover, $f(0)\in[0,1]$ and $f(1)=2$.  If $p_\infty$ is a valid probability function, then $f(0)=0$.

\begin{prop}\label{P:pi} If $\mF = \{T_n\subset G(n)\mid n\in\N\}$ is a convergent family and $f$ is its associated power series, then for each $x\in(0,1]$,
$$ \lim_{n\ra\infty}\left(\frac{\tP_{\lfloor x M_n\rfloor}}{M_n}\right) = f(x),$$
where $M_n = n^2-1 = |T_n|$.
\end{prop}

\begin{proof}
Let $x\in(0,1]$.  Consider a fixed large value of $n$, and interpret all running notation with respect to the spanning tree $T=T_n$.  Set $i=\lfloor x M\rfloor$ so that $\frac{M-i}{M}\approx 1-x$.

The tail $\sum_{d\geq D}p_\infty(d)$ can be made arbitrarily small by choosing $D$ large.  Furthermore, $n$ can be chosen large enough so that $p(d)$ is arbitrarily close to $p_\infty(d)$ for all $d\leq D$, where $p$ denotes the degree mass function of $T=T_n$.  In other words, the degree mass function of $T$ becomes arbitrarily close to $p_\infty$ where it matters.  This observation and the fact that $M\approx N$ give the following.
\begin{align*}
\tilde{P}_i
  & \approx i + M\left( 1 - \sum_{d\geq 3}p_\infty(d)\left(\frac{M-i}{M}\right)^d\right) \\
  & = M\cdot\left( 1+\frac{i}{M}- \sum_{d\geq 3} p_\infty(d)\left(\frac{M-i}{M}\right)^d\right)
   \approx M\cdot f\left(\frac{i}{M}\right)\approx M\cdot f(x).
\end{align*}

\end{proof}
The proposition essentially says that the sequence $\{f(i/M)\mid i=1,2,...,M\}$ approximates the scaled approximate passing times $\{\tP_i/M\mid i=1,2,...,M\}$.

\begin{lem}\label{L:var} If $\mF=\{T_n\subset G(n)\mid n\in\N\}$ is a bounded, convergent, neighbor-independent family, then for every $x\in(0,1]$,
$$\lim_{n\ra\infty}\E \left(\frac{P_{\lfloor x M_n\rfloor}(\alpha)}{M_n}\right) = f(x),\text{ and }
\lim_{n\ra\infty}\Var\left(\frac{P_{\lfloor x M_n\rfloor}(\alpha)}{M_n}\right)= 0,$$
where $M_n = n^2-1 = |T_n|$, and $f$ is the power series associated to $\mF$.
\end{lem}
\begin{proof}
The first equation just combines Proposition~\ref{P:Approxpasstimes} and Proposition~\ref{P:pi}.

For the second equation, consider a fixed large value of $n$, and interpret all running notation with respect to the spanning tree $T=T_n$.  Set $i = \lfloor x M\rfloor$.

If $\alpha\in S(T)$ is uniformly selected, then $\V_1^i = \{e_1,...,e_i\}\subset\V_1$ is a random sample from $\V_1$.  Since $P_i(\alpha)=i + |\V_2^i|$, it will suffice to bound the variance of $\frac{|\V_2^i|}{M}$.  For this, we write
$$|\V_2^i| = \sum_{v\in\V_2} Y_v,$$
where $Y_v$ is the indicator random variable; that is, $Y_v\in\{0,1\}$ depending on whether $v\in\V_2^i$.

Notice that $\text{Cov}(Y_v,Y_w)\leq 1$ for all $v,w\in\V_2$.  Furthermore, if $v$ and $w$ have no common neighbors, then $\text{Cov}(Y_v,Y_w)\leq 0$; that is, the knowledge that $Y_v=1$ makes it slightly less likely that $Y_w=1$.

Like in Definition~\ref{D:neighborcon}, define $\Delta=\V_2\times\V_2$ and let $\tilde\Delta$ be the set of all $(v,w)\in\Delta$ such that $v\neq w$ and $v,w$ have at least one common neighbor in $\V_1$.  
\begin{align*}
\text{Var}\left( |\V_2^i|\right) 
  &= \sum_{v,w\in\V_2}\text{Cov}(Y_v,Y_w)
  \leq \sum_{v\in\V_2} \text{Var}(Y_v) + \sum_{(v,w)\in\tilde{\Delta}}\text{Cov}(Y_v,Y_w) 
  \leq |\V_2|+|\tilde{\Delta}|.
\end{align*}
Since $\mF$ is neighbor-independent, we have
\begin{align*}
\text{Var}\left( \frac{|\V_2^i|}{M}\right)
  & = \frac{\text{Var}\left( |\V_2^i|\right)}{M^2}\approx 0
\end{align*}
\end{proof}

\section{The geometric mean of the scaled passing times}
In this section, we define and study the following measurement, which is the geometric mean of the scaled passing times.
\begin{definition}\label{D:A} For $T\subset\mathcal{T}(G(n))$,
$$ \mA(T) = \frac{1}{M}\left(\prod P_i(\alpha) \right)^{\frac 1 M} = \left(\prod \frac{P_i(\alpha)}{M} \right)^{\frac 1 M},$$
which is a random variable with respect to a uniform choice of $\alpha\in S(T)$.
\end{definition}

The support of $\mA(T)$ lies in $[0,2]$ because $P_i(\alpha)/M\in[0,2]$ for each $i$ and each $\alpha$.

\begin{prop}\label{P:dodo} If $\mF=\{T_n\subset G(n)\mid n\in\N\}$ is a family of spanning trees, then
$$Q^-(\mF) \geq \frac{1}{e\cdot\liminf_{n\ra\infty}\left(\E \left(\mA(T_n)\right)\right)}.$$
In other words,
$$\PMST(T_n)>\sim\left( e\cdot\E(\mA(T_n))\right)^{-n^2}.$$
\end{prop}

\begin{proof}
Consider a fixed large value of $n$, and interpret all running notation with respect  to the spanning tree $T=T_n$.  Proposition~\ref{P:russ} together with Stirling's approximation and Jensen's inequality yields:
\begin{align*}
\PMST(T)
    & = M!\cdot \E\left( \frac{1}{\prod P_i(\alpha)}\right)
             \approx \left(\frac{M}{e}\right)^M \cdot \E\left( \frac{1}{\prod P_i}\right) \\
    & =  e^{-M}\cdot \E\left( \frac{1}{\prod\frac{P_i}{M}}\right)
      = e^{-M}\cdot\E\left(\mA(T)^{-M}\right)\\
    & \geq e^{-M}\cdot\left(\E(\mA(T))\right)^{-M}\sim\left( e\cdot\E(\mA(T))\right)^{-n^2}.
\end{align*}
\end{proof}

To interpret the following proposition, recall that the \emph{geometric mean} of a function $f:[0,1]\ra\R^+$ is defined as
$$\overline{f} = \exp\left(\int_0^1 \ln f(x)\,\text{dx} \right).$$  This is because a Riemann sum for $\int_0^1 \ln f(x)\,\text{dx}$ is the log of the geometric mean of the sample values.

\begin{prop}\label{P:lulu}
If $\mF=\{T_n\subset G(n)\mid n\in\N\}$ is a bounded, convergent, neighbor-independent family, and $f$ is its associated power series, then
$$ \lim_{n\ra\infty} \E\left(\mA(T_n)\right)= \overline{f}.$$
\end{prop}

\begin{proof}
For each $n\in\N$ and each $\alpha\in S(T_n)$, let $g_n^\alpha:[0,1]\ra[0,2]$ denote the piecewise linear function whose graph connects the following points:
$$\left\{\left(0,0\right),\left(\frac 1 M,\frac{P_1(\alpha)}{M}\right),
\left(\frac 2 M,\frac{P_2(\alpha)}{M}\right),...,\left(1,\frac{P_M(\alpha)}{M}\right)\right\},$$
where $M=|T_n|$ as usual.

Notice that for every $x\in[0,1]$,
\begin{equation}\label{E:helloworld}
g_n^\alpha(x)\geq x
\end{equation}

We claim that for every $\epsilon>0$ there exists $n_0\in\N$ such that if $n>n_0$ then
\begin{equation}\label{ear}
\max\{|g_n^\alpha(x)-f(x)| \mid \epsilon\leq x\leq 1\} <\epsilon \text{ with probability } >1-\epsilon,
\end{equation}
where ``probability'' means with respect to a uniform choice of $\alpha\in S(T_n)$.
That is, with high probability, the graph of $g_n^\alpha$ lies within a small tubular neighborhood of the graph of $f$ over most of its domain.

To prove this, let $\epsilon>0$.  Define $C=\max\{f'(x)\mid 0\leq x\leq 1\}$ and choose $L\in\N$ with $L>\frac{5C}{\epsilon}$.  For each $j\in\{1,2,...,L\}$ define $x_j = \frac{j}{L}$, so that the points $\{x_j\}$ are equally spaced from $0$ to $1$ with gaps $<\frac{\epsilon}{5C}$.

For a fixed $j$, applying Lemma~\ref{L:var} to $x_j$ establishes that there exists $n_0\in\N$ such that for all $n>n_0$,
\begin{equation}\label{fish}
|g_n^\alpha(x_j)-f(x_j)|<\frac{\epsilon}{5} \text{ with probability }>1-\epsilon.
\end{equation}
In fact, there exists $n_0\in\N$ such that for all $n>n_0$, Equation~\ref{fish} is true \emph{for all} $j\in\{1,2,...,L\}$.

But if $|g_n^\alpha(x_j)-f(x_j)|<\frac{\epsilon}{5}$ for all $j\in\{1,2,...,L\}$, then we claim that $|g_n^\alpha(x)-f(x)|<\epsilon$ for all $x\in[\epsilon,1]$.

To see this, let $x\in[\epsilon,1]$.  Since $x\geq \epsilon$, we know $x>x_1$, so it is possible to choose an index $j$ such that $x_j<x<x_{j+1}$.  By the triangle inequality,
\begin{align*}
|g_n^\alpha(x)-f(x)|
  & \leq |f(x)-f(x_j)| + |f(x_j)-g_n^\alpha(x_j)| + |g_n^\alpha(x_j)-g_n^\alpha(x)| \\
  & \leq \frac{\epsilon}{5}+\frac{\epsilon}{5}+\frac{3\epsilon}{5}=\epsilon.
\end{align*}
To justify the claim here that $|g_n^\alpha(x_j)-g_n^\alpha(x)|<\frac{3\epsilon}{5}$, we use the fact that $g_n^\alpha$ is increasing, so
\begin{align*}
|g_n^\alpha(x_j)-g_n^\alpha(x)| & \leq |g_n^\alpha(x_j)-g_n^\alpha(x_{j+1})| \\
              & \leq |g_n^\alpha(x_j)-f(x_j)| + |f(x_j)-f(x_{j+1})| + |f(x_{j+1})-g_n^\alpha(x_{j+1})|\\
              & \leq \frac{\epsilon}{5}+\frac{\epsilon}{5}+\frac{\epsilon}{5}=\frac{3\epsilon}{5}.
\end{align*}
This completes the proof of Equation~\ref{ear}.

In summary, for any $\epsilon>0$, we can choose $n$ sufficiently large so that the probability is $\geq 1-\epsilon$ that the graph of $g_n^\alpha$ lies within an $\epsilon$ tubular neighborhood of the graph of $f$ on $[\epsilon,1]$.  If this event occurs then we have the following.
\begin{align*}
\ln\left(\mA(T_n)(\alpha)\right)
  & = \frac{1}{M}\sum \ln\left(\frac{P_i(\alpha)}{M}\right)\\
  & = \frac{1}{M}\sum_{i<\epsilon M} \ln\left(\frac{P_i(\alpha)}{M}\right)
      +\frac{1}{M}\sum_{i\geq\epsilon M} \ln\left(\frac{P_i(\alpha)}{M}\right)\\
  & \leq \epsilon\ln(2) + \int_\epsilon^1\ln(f(x)+\epsilon)\,\text{dx}\\
  & = \epsilon\ln(2) + \epsilon(1-\epsilon) + \int_\epsilon^1\ln(f(x))\,\text{dx}\\
  & \leq \epsilon\ln(2) + \epsilon(1-\epsilon) + \int_0^1\ln(f(x))\,\text{dx} - \int_0^\epsilon \ln(x)\,\text{dx} \\
  & = \underbrace{\epsilon(\ln(2)-\ln(\epsilon)) -\epsilon^2}_{\text{denoted }\psi(\epsilon)} + \int_0^1\ln(f(x))\,\text{dx},
\end{align*}
The first inequality above uses the fact that each $\frac{P_i(\alpha)}{M}\leq 2$, while the second uses the fact that $f(x)\geq x$.

We therefore have the following.
$$\mA(T_n)(\alpha)\leq e^{\psi(\epsilon)}\cdot\overline{f}.$$
Since this occurs with probability $\geq (1-\epsilon)$, and since the support of $\mA$ is contained in $[0,2]$,
\begin{align*}
\E(\mA(T_n))
  & \leq \epsilon\cdot 2 + (1-\epsilon)\cdot e^{\psi(\epsilon)}\cdot\overline{f}
\end{align*}
Since $\lim_{\epsilon\ra 0}\psi(\epsilon)=0$, this suffices to prove that
$$\lim_{n\ra\infty} \E\left(\mA(T_n)\right)\leq \overline{f}.$$
The corresponding lower bound is proven analogously, which requires Equation~\ref{E:helloworld} and the fact that $f(x)\leq 2$.

\end{proof}

Combining Proposition~\ref{P:dodo} and Proposition~\ref{P:lulu} proves Theorem~\ref{T_pre}, which we restate as follows.
\begin{prop}\label{P:punchline}
If $\mF=\{T_n\subset G(n)\mid n\in\N\}$ is a bounded, convergent, neighbor-independent family, and $f$ is its associated power series, then
$$Q^-(\mF) \geq \frac{1}{e\cdot\overline{f}}.$$
In other words,
$$\PMST(T_n)>\sim\left( e\cdot\overline{f}\right)^{-n^2}.$$
\end{prop}

\section{Example families of spanning trees}
In this section, we study some families of spanning trees, including the families pictured in Figure~\ref{F:families}.

\begin{figure}[ht!]\centering
   \scalebox{.9}{\includegraphics{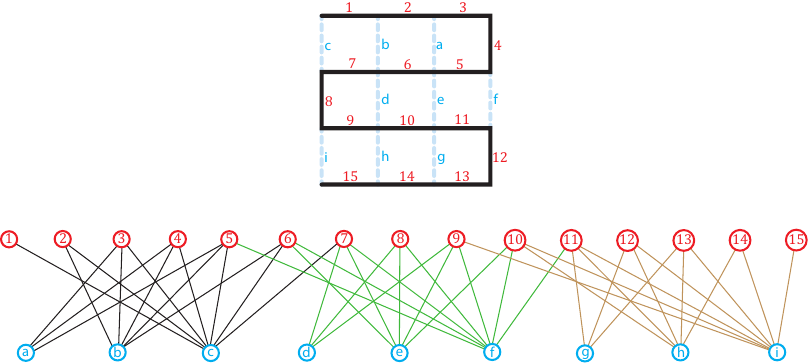}}
\caption{The bipartite graph of the centipede $T_4\subset G(4)$.}\label{F:cc}
   \end{figure}

\begin{example}[The Centipede Family]
We first study the centipede family $\mF = \{T_n\subset G(n)\mid n\geq 3\}$.  The bipartite graph of $T_4$ is shown in Figure~\ref{F:cc}. The degree-mass function $p_n$ of the centipede $T_n$ is particularly simple:
\begin{equation}\label{E:cent_deg} p_n(d) = \frac{1}{n-1} \text{ for each }d\in\{3,5,...,2n-1\}.
\end{equation}
Therefore for each $d\geq 3$, $p_\infty(d)=\lim_{n\ra\infty}p_n(d)=0$, so $\mF$ is convergent and its power series is
$$f(x) = 1+x.$$
This power series encodes the fact that as $n\ra\infty$, the approximate passing times for $T_n$ look more and more like the extreme case considered in the proof of the lower bound of Theorem~\ref{T:bound2}; namely $\{\tilde{P}_i\}\approx\{N,N+1,...,N+M\}$.

The centipede family is bounded because Equation~\ref{E:cent_deg} says that the maximum degree in $\V_2$ equals $2n-1$, which becomes arbitrarily small relative to $n^2$.  The centipede family is also neighbor-independent, as can be seen from the structure of the bipartite graph exemplified in Figure~\ref{F:cc}.  A pair of nodes of $\V_2$ share a neighbor if and only if the corresponding chords are in the same or adjacent rows.

Proposition~\ref{P:punchline} only provides the following old information.
$$ Q^-(\mF)\geq \frac{1}{e\cdot \overline{f}} = \frac{1}{4}.$$
We do not know whether $Q(\mF)=\frac 1 4$.
\end{example}

\begin{example}[The Double Spiral Family]
The double spiral family  in Figure~\ref{F:families} is convergent and has the same power series as the centipede, namely $f(x) = 1+x$.  This family is not bounded because the average degree grows proportional to $M$.  In fact, the family is designed to have high average degree growth, and hence be a good candidate for having the minimum possible decay base.
\end{example}

\begin{example}[The Fractal Family]
We next study the fractal family $\mF = \{F_k\subset G(2^k)\mid k=1,2,...\}$.  It is bounded because the average degree of nodes in the bipartite graph of $F_k\subset G(2^k)$ grows logarithmically in $n=2^k$, as shown in~\cite{Alon}.

Let $p_k$ denote the degree-mass function for $F_k$.  Unlike the mass functions for the centipede family, these mass functions converge to non-zero limits.

\begin{lem} For each $d\geq 3$, the sequence $\{p_k(d)\}$ converges to a limit, which we denote as $p_\infty(d)$, and $\sum_{k\geq 3} p_\infty(k) = 1$.
\end{lem}

\begin{proof}
Let $q_k(d)$ denote the \emph{number} of degree $d$ nodes in $\V_2$, so
\begin{equation}\label{E:recur}
p_k(d) = \frac{q_k(d)}{|\V_2|} = \frac{q_k(d)}{2^{2k}-2^{k+1}+1}.
\end{equation}

Fix a value $d\geq 3$.  Assume it's not the case that $q_k(d)=0$ for all $k$.  We claim there exists an index $k_0$ beyond which $\{q_k(d)\}$ satisfies a simple first-order linear recurrence relation with base $4$; that is:
$$q_{k+1}(d) = 4\cdot q_k(d) + C\,\,\,\text{ for all } k\geq k_0.$$
The solution is:
$$q_{k_0 + t} = \left(A+\frac{C}{3}\right)4^t - \frac{C}{3}\,\,\,\text{ for all } t\geq 0,$$
where $A = q_{k_0}(d)$.
Therefore Equation~\ref{E:recur} gives:
\begin{equation}\label{E:limp}\lim_{t\ra\infty}p_k(d) = \frac{A+\frac{C}{3}}{4^{k_0}}.
\end{equation}

To justify the recurrence relation, notice that each branch of $F_{k+1}$ either belongs to one of the four copies of $F_k$ out of which $F_{k+1}$ is built, or is a ``cut branch'' that connects two of these quadrants.  Thus,
$$q_{k+1}(d) = 4\cdot q_k(d) + C_{k+1}(d),$$
where $C_{k+1}(d)$ is the number of cut branches with degree $d$.  We claim that for sufficiently large $k$, $C_{k+1}(d)$ doesn't depend on $k$.  To see this, consider the list of degrees of nodes in the set $\V_2$ from $F_{k+1}$ corresponding to cut branches.  This list is naturally partitioned into two halves of equal length.  The inner half is exactly the list from $F_k$, whereas the minimal degree for the outer half goes to infinity as $k\ra\infty$.

From the above calculations, one can prove that the convergence is uniform: for every $\epsilon>0$ there exists $d_0$ such that $\sum_{d\geq d_0}p_k(d)<\epsilon$ for all $k\in\N$.  From this, it is straightforward to conclude that $\sum_{d\geq 1}p_\infty(d)=1$.
\end{proof}

Table~\ref{T:T1} reports the value $p_\infty(d)$ for all $3\leq d\leq 125$, which we computed using Equation~\ref{E:limp} and a computer.

The Fractal family $\mF$ is neighbor-independent for the following reason.  Define $\Delta_k=\V_2\times\V_2$ (in the bipartite graph associated to $T_k$), and let  $\tilde\Delta_k$ be the set of pairs $(v,w)\in\Delta_k$ that have at least one common neighbor in $\V_1$.  Due to the recursive definition of the family,
$$\frac{|\tilde\Delta_k|}{|\Delta_k|}\approx \frac 1 4 \cdot \frac{|\tilde\Delta_{k-1}|}{|\Delta_{k-1}|}.$$
This is because chords from different quadrants have no common neighbors, and the contribution of the cut edges between the quadrants becomes negligible as $k\ra\infty$.  Therefore $\frac{|\tilde\Delta_k|}{|\Delta_k|}\ra 0$.

\begin{table}\label{T:T1}\caption{$p_\infty(d)$ for all $3\leq d\leq 125$ for the Fractal family}
\centering
\begin{tabular}{c||c|c|c|c|c|c|c|c|c|}
$d$ & $3$ & $5$ & $11$ & $13$ & $25$ & $27$ & $29$ & $55$ & $57$  \\
\hline
$p_\infty(d)$ & $\frac{5}{12}$ & $\frac{1}{4}$ & $\frac{1}{12}$ & $\frac{1}{12}$ & $\frac{1}{96}$
 & $\frac{1}{24}$ & $\frac{1}{32}$ & $\frac{1}{384}$ & $\frac{1}{128}$
\end{tabular}
\end{table}
\begin{table}
\centering
\begin{tabular}{c||c|c|c|c|c|c|c|}
$d$ & $59$ & $61$ & $117$ & $119$ & $121$ & $123$ & $125$ \\
\hline
$p_\infty(d)$ &  $\frac{7}{384}$
 & $\frac{5}{384}$  & $\frac{1}{1536}$ & $\frac{1}{384}$ & $\frac{1}{256}$ & $\frac{1}{128}$ & $\frac{3}{512}$
\end{tabular}
\end{table}

Let $f$ be the power series for $\mF$.  Let $f^+$ denote polynomial approximations of $f$ obtained by truncating after the $x^{125}$ term.  $\mF$ is bounded and convergent, so we can apply Proposition~\ref{P:punchline}, which yields the following.
$$(Q^-(\mF))^{-1}\leq e\cdot\overline{f}\leq e\cdot\overline{f^+}  = 3.2508\cdots.$$
Adding more terms to $f^+$ improves this bound, but only slightly; in fact, it can be shown that
$e\cdot\overline{f}>\mathfrak{b}=3.2099\cdots$.
Since $\sum p_\infty(d)=1$, this can be proven by choosing enough terms to make the sum close to $1$ and hence bounding the tail. 
\end{example}

\begin{example}[The Uniform Family] Consider a family $\mF = \{T_n\subset G(n)\mid n\in\N\}$ where each spanning tree $T_n$ is uniformly randomly chosen.  We'll call $\mF$ a \emph{uniform family}.  Let $p_n$ denote the degree-mass function of $T_n$.  Although the randomness means that $T_n$ and $p_n$ are not uniquely determined, the limit $p_\infty = \lim(p_n)$ is uniquely determined with probability $1$ because it can be interpreted as the degree-mass function for a uniformly random spanning tree of the infinite lattice $\Z^2$.  Manna et al. studied this situation in~\cite{Dhar} and proved that $p_\infty(3)\approx 0.29454$ and $p_\infty(5)\approx 0.12409$ (these are decimal approximations of the exact values that they calculated).  The also provided strong numerical and theoretical evidence for their conjecture that $d\mapsto p_\infty(d)$ decays as $d^{-8/5}$ for large $d$.

This conjecture indicates that the following is a reasonable approximation of the power series $f$ of $\mF$:
$$f(x)\approx 2-x-(0.29454)x^3 - (0.12409)x^5 - C\cdot\sum_{d\geq 7\text{ odd }}d^{-8/5},$$
where $C$ is chosen to make it a valid probability mass function.  This yields:
\begin{equation}\label{E:Univ_bound}
e\cdot\overline{f}= 3.433\cdots.
\end{equation}
This is higher than the corresponding value for the fractal family, which is evidence that the fractal family has higher MST probability than does a uniform family, although we do not know whether a uniform family is necessarily bounded or neighbor-independent with probability tending to one, or whether all uniform families necessarily have the same decay base.
\end{example}

Observe that a uniform family behaves qualitatively like fractals, not like centipedes, in the sense that $p_\infty$ is a valid probability function.

\section{Questions and Conjectures}
A natural open question is to identify the optimal upper and lower bounds in Theorem~\ref{T:bound2}, which will require the construction of a family of spanning trees that achieves each bound.  A potential starting point for addressing the lower bound is to answer the following.
\begin{question} Is the decay base of the double spiral family equal to $\frac 1 4$?
\end{question}

The fractal family $\mF$ might not achieve the optimal upper bound, but it is reasonable to expect it to come close.  This family was originally believed to be optimal for the related problem of minimizing average stretch, but families were later constructed that were slightly better~\cite{Koh},\cite{Lieb}.  Proposition~\ref{P:punchline} only provided a lower bound for the decay base of the fractal family, but this bound isn't sharp, so the following question remains.
\begin{question} What is the decay base of the fractal family?
\end{question}

One way to address this is to quantify how close the inequality in Proposition~\ref{P:punchline} comes to being an equality.  Here is a concise overview of the proof of Proposition~\ref{P:punchline}.
\begin{align*}
\PMST(T_n)
  & \sim e^{-M}\cdot\E\left(\mA(T_n)^{-M} \right)
    \geq e^{-M}\cdot\left(\E(\mA(T_n))\right)^{-M}\\
  & \sim e^{-M}\cdot \overline{f}^{-M} \sim \left(e\overline{f}\right)^{-n^2}
\end{align*}
The only inequality here comes from Jensen's Inequality, so it would be necessary to understand the Jensen Gap, which requires understanding how quickly the variance of $\mA(T)$ decays with $M$.  The following conjecture quantifies this.

\begin{conjecture} If $\mF=\{T_n\subset G(n_n)\mid n\in\N\}$ is a bounded, convergent, neighbor-independent family, and $f$ is its associated power series, then
$$Q(\mF) = \frac{\exp\left( \lim_{n\ra\infty} \frac{(n\cdot\sigma_n)^2}{2}\right)}{e\cdot \overline{f}},$$
where $\sigma_n^2$ denotes the variance of the random variable $X_n=\ln(\mA(T_n))$.
\end{conjecture}

\begin{proof}[Sketch of potential proof idea] Consider a fixed large value of $n$, and interpret all running notation with respect to the spanning tree $T=T_n$.  Set $X=X_n$.  The proof of Proposition~\ref{P:lulu} establishes that
$$\E(X)\approx\ln\overline f.$$
Let $\Psi$ be the moment generating function of $X$.  Stirling's approximation gives:
\begin{align*}
\Psi(-M)
  & = \E\left(e^{-MX}\right) = \E\left(\prod\frac{M}{P_i(\alpha)}\right)
   = M^M\cdot\E\left(\prod\frac{1}{P_i(\alpha)}\right)\\
  & \approx e^M\cdot M!\cdot\E\left(\prod\frac{1}{P_i(\alpha)}\right) = e^M\cdot\PMST(T).
\end{align*}

On the other hand, if we assume that $X$ becomes normal as $n\ra\infty$, then
$\Psi(t) \approx \exp\left(\E(X) t + \frac{\sigma^2}{2}t^2\right),$ so

\begin{align*}
\Psi(-M) &\approx \frac{e^{(\sigma\cdot M)^2/2}}{\overline{f}^M}
\end{align*}
Combining the previous two equations yields
\begin{align*}
\PMST(T)
   \approx\frac{e^{(\sigma\cdot M)^2/2}}{e^M\cdot \overline{f}^M}
    = \left(\frac{e^{\sigma^2 M/2}}{e\cdot\overline{f}}\right)^M,
\end{align*}
from which the result follows.
\end{proof}


\bibliographystyle{amsplain}

\begin{thebibliography}{9}

\bibitem{complete} L. Addario-Berry, N. Broutin, and B. Reed. \emph{Critical random graphs and the structure of a minimum spanning tree}. Random Structures Algorithms, 35(3):323–347, 2009.

\bibitem{Alon} N. Alon, R. Karp, D. Peleg, D. West, \emph{A Graph-Theoretic Game and its Application to the $k$-Server Problem}, Sian J. Comput., Vol 24, No. 1 pp. 78-100, February 1995.

\bibitem{Models} E. Babson, M. Duchin, A. Iseli, P. Poggi-Corradini, D. Thurston, J. Tucker-Foltz, \emph{Models of random spanning trees} (2024), https://arxiv.org/pdf/2407.20226.

\bibitem{Rev_ReCom}Sarah Cannon, Moon Duchin, Dana Randall, Parker Rule, \emph{Spanning tree methods for sampling graph partitions}, arXiv:2210.01401 (Oct. 2022).

\bibitem{revRecom} S. Cannon, M. Duchin, P. Rule, D. Randall, \emph{A reversible recombination chain for graph partitions}, preprint.

\bibitem{MergeSplit} Daniel Carter, Gregory Herschlag, Zach Hunter, Jonathan Mattingly, \emph{Metropolized forest recombination for Monte Carlo sampling of graph partitions}, SIAM J. Appl. Math. 83 (2023), no. 4, pp. 1366–1391.

\bibitem{ReCom} Daryl DeFord, Moon Duchin, Justin Solomon, \emph{Recombination: A Family of Markov Chains for Redistricting}, Harvard Data Science Review 3.1 (Dec. 2020).

\bibitem{KLRW} E. K\"ohler, C. Liebchen, G. W\"unsch, R. Rizzi, \emph{Lower Bounds for Strictly Fundamental Cycle Bases in Grid Graphs}.

\bibitem{Koh} E. K\"ohler, R. Rizzi, G. W\"unsch, \emph{Reducing the optimality gap of strictly fundamental cycle bases in planar grids}, Preprint 007/2006, TU Berlin, Mathematical Institute, 2006.

\bibitem{stack}Lavrov, Misha, \emph{Variance of a sample from a bipartite graph}, Stack Exchange solution, https://math.stackexchange.com/q/4848665.

\bibitem{Lieb} C. Liebchen, G. W\"unch, E. K\"ohler, A. Reich, R. Rizzi, \emph{Benchmarks for strickly fundamental cycle bases} (2007)

\bibitem{Russ_book}R. Lyons, Y. Peres. \emph{Probability on trees and networks}, volume 42 of Cambridge Series in Statistical and Probabilistic Mathematics. Cambridge University Press, New York, 2016.

\bibitem{Russ} R. Lyons, Y. Peres, O. Schramm, \emph{Minimal spanning forests}, Ann. Probab., \textbf{34}(5), 2006, pp. 1665--1692.

\bibitem{Dhar} M. Manna, D. Dhar, S. Majumdar,  \emph{Spanning-Trees in 2 Dimensions}, Physical Review A, Atomic, Molecular and Optional Physics, October 1992.

\bibitem{SMC} Cory McCartan, Kosuke Imai, \emph{Sequential Monte Carlo for sampling
balanced and compact redistricting plans}, Annals of Applied
Statistics 17.4 (Dec. 2023), pp. 3300–3323.

\bibitem{Wu}F. Wu, \emph{Number of spanning trees on a lattice}. J. Phys. A, 10(6):L113–L115, 1977.

\end{thebibliography}

\end{document}